\documentclass[11pt,french,english]{article}

\usepackage{amsmath,amssymb,amsthm,amstext,amsfonts} 
\usepackage{mathrsfs,bm,bbm,enumerate,mathabx,mathtools}
\usepackage{graphicx,color,tikz}
\usepackage{babel}
\usepackage{soul}
\usepackage[T1]{fontenc} 
\usepackage[latin9]{inputenc}

\newtheorem{example}{\textbf{Example}}

\newtheorem{corollary}{\textbf{Corollary}}
\newtheorem{proposition}{\textbf{Proposition}}
\newtheorem{theorem}{\textbf{Theorem}}

\theoremstyle{definition}
\newtheorem{definition}{\textbf{Definition}}
\newtheorem{notation}{\textbf{Notation}}
\newtheorem{remark}{\textbf{Remark}}

\providecommand{\abs}[1]{\left\lvert#1\right\rvert}
\providecommand{\norm}[1]{\left\rVert#1\right\rVert}
\providecommand{\set}[1]{ \left\{ #1  \right\}  }
\providecommand{\setb}[2]{ \left\{ #1 \ \middle| \  #2 \right\}  }
\providecommand{\innprod}[2]{\left\langle #1, #2 \right\rangle}
\providecommand{\parenth}[1]{\left( #1 \right) }

\def\CF{{\widehat{\mathscr{P}}}}
\def\D{{\mathcal{D}}}
\def\S{{\mathcal{S}}}
\def\R{{\mathcal{R}}}
\def\Lop{\mathrm{L}} 
 
\def\C{ \mathbb{C}}
\def\Z{ \mathbb{Z}}

\def\N{ \mathbb{N}}
\def\R{ \mathbb{R}}

\def\drm{\mathrm{d}}

\def\ui{\mathrm{i}}

\def\ue{\mathrm{e}}
\def\Der{\mathrm{D}}
\def\FL{(-\Delta)^{\gamma/2}}
\def\T{\mathbb{T}}

\def\lexp{\psi}

\def\sas{S$\alpha$S}
\def\nn{n}
\def\waveexp{\zeta}

\date{}

\title{The \MakeLowercase{$\nn$}-term Approximation of Periodic Generalized L\'evy Processes 
\thanks{This work is an extension of the conference paper \cite{ward15}.}
}

\author{
Julien Fageot 
\thanks{Biomedical Imaging Group, \'Ecole polytechnique f\'ed\'erale de Lausanne (EPFL),
Station 17, CH-1015, Lausanne, Switzerland}
\and
Michael Unser 
\footnotemark[2]
\and
John Paul Ward 
\thanks{North Carolina A\&T State University,
Department of Mathematics,
Marteena Hall,
1601 East Market Street,
Greensboro, NC 27411}
}

\begin{document}

\maketitle

\begin{abstract}
\ In this paper, we study the compressibility of  random processes and fields, called generalized L\'evy processes,  that are solutions of stochastic differential equations driven by  
$d$-dimensional periodic
L\'evy white noises. 
Our results are based on the estimation of the Besov regularity of L\'evy white noises and generalized L\'evy processes.
We show in particular that non-Gaussian generalized L\'evy processes are more compressible in a wavelet basis than the corresponding Gaussian processes, in the sense that their $\nn$-term approximation errors decay faster. We quantify this compressibility in terms of the Blumenthal-Getoor indices of the underlying L\'evy white noise.
\end{abstract}

%
%

\section{Introduction} \label{sec:intro}

Stochastic models are commonly used in engineering and financial applications  \cite{davis92,kidmose00,nikias95,rabiner89}. The most widely considered stochastic models are  
Gaussian. 
However,  the Gaussian assumption is too restrictive for many applications, and more general models are needed to accurately represent real data.  
Stochastic differential equations driven by 
L\'evy noises are able to play this role. 
We call the solutions of these equations \emph{generalized L\'evy processes}.
Generalized L\'evy processes encompass Gaussian models and 
the family of compound-Poisson processes that are pure jump processes. 
They also include symmetric-alpha-stable (\sas) processes which maintain many of the   desirable properties of Gaussian models; 
for example, 
they satisfy a  generalized central-limit theorem \cite{samoradnitsky94,Fageot2016scaling}.

Generalized L\'evy processes with no Gaussian component are particularly relevant to applications \cite{shao93}.
They are called \emph{sparse stochastic processes} \cite{Fageot2014,Unser2014sparse} 
due to their ability to model sparse signals \cite{Amini2014sparsity,Bostan.etal2013}.
Each member of the family of generalized L\'evy processes is associated with  two parameters $0\leq \beta' \leq \beta \leq 2$, the Blumenthal-Getoor indices. For most of the generalized L\'evy processes, including Gaussian, S$\alpha$S, and compound Poisson, these indices are equal, hence we only refer to $\beta$ in this introduction. 
As we shall see, $\beta$
determines the sparsity level of the generalized L\'evy process in a wavelet basis. Here, we characterize sparsity by $n$-term wavelet approximation rates; in other words, a process is said to be sparser than another one if it satisfies a faster rate of  decay of its error of approximation. 

Our interest in a sparse wavelet representation is compressibility, since it allows one to  store and transmit processes more efficiently by first expanding them in a wavelet basis.
We shall show that Gaussian processes (for which $\beta= \beta' = 2$) are the least compressible, 
while  the compressibility of non-Gaussian processes increases as 
$\beta$ decreases.  
Compressibility  is also a useful assumption in inverse imaging problems. If an image is known to be sparse a priori, then this information can be incorporated to produce a better reconstruction algorithm \cite{bostan13,elad10,Unser2014sparse}.

Our model is the stochastic differential equation
\begin{equation}\label{eq:ls=w}
	\mathrm{L}s
	=
	w,
\end{equation}
where $\mathrm{L}$ is a differential operator, $s$ is a stochastic process, and $w$ is a L\'evy white noise. The two defining components are the operator and the white noise. 
The white noise $w$ is a random periodic generalized function.
We consider standard differential-type operators of the following type: $\mathrm{L}$ is assumed to be an order $\gamma>0$ operator
(typically an order $\gamma$ derivative)
 that reduces the regularity of a function by $\gamma$. 

The compressibility of a given function---and, by extension, of a stochastic process---can be quantified in terms of its Besov regularity. Besov spaces are important function spaces that can be characterized precisely by 
the rate of decay of their $\nn$-term approximation error
\cite{cohen00rna,garrigos04}. 
Our strategy to access the compressibility of a stochastic process is to connect results on the Besov regularity of the underlying white noise to compressibility by (deterministic) approximation-theoretic arguments.

The question of the Besov regularity of L\'evy processes---that corresponds to \eqref{eq:ls=w} in dimension $d=1$ with 
$\Lop = \mathrm{D}$ 
(the derivative operator)---has been addressed in the literature for subfamilies of L\'evy processes \cite{benyi2011modulation,ciesielski1993quelques,roynette1993mouvement} as well as in the general case \cite{herren1997levy}. Several extensions by R. Schilling have been obtained for L\'evy-type processes \cite{schilling1997Feller,Schilling1998growth,Schilling2000function}; see also \cite[Chapter 5]{Bottcher2013levy} for a summary. In our case, we aim at considering a multidimensional setting with a general operator $\Lop$. This work is therefore a continuation of our previous works on the Besov regularity of L\'evy white noises \cite{FUW2015Besov,Fageot2016multidimensional,Aziznejad2018sharp}.
 
The rest of the paper is organized as follows: 
In Section \ref{sec:maths}, we cover the mathematical background for the remaining sections and  discuss function spaces and random processes.  
In Section \ref{sec:sparse}, we define a class of admissible operators $\Lop$ for \eqref{eq:ls=w}, which leads to a definition of generalized L\'evy processes on the torus.
In Section \ref{sec:Nterm}, we put forth the main results of the paper: $\nn$-term approximation of periodic generalized L\'evy processes.
Finally, we conclude in Section \ref{sec:discuss}  with a  discussion of our results.


\section{Mathematical Background} \label{sec:maths}

Motivated by the study of local properties of solutions of stochastic differential equations,
this paper deals with periodic random processes.
Therefore, we consider spaces of periodic functions in the sequel. 
We specify these functions on their fundamental domain,
the 
$d$-dimensional torus 
$\mathbb{T}^d
= 
[-1/2,1/2)^d 
\subset 
\mathbb{R}^d $, 
$d\geq 1$.


\subsection{Lebesgue and  Sobolev Spaces}
\label{subsec:sobolev}

The space of continuous functions on the torus is 
$C(\mathbb{T}^d)$, and
$C^k(\mathbb{T}^d)$
denotes the functions with $k\in\mathbb{N}$ continuous derivatives.
The Lebesgue space $L_p(\mathbb{T}^d)$ is the collection of measurable functions for which 
\begin{equation} \label{eq:lebesgue}
	\norm{f}_{L_p(\mathbb{T}^d)}
	:=
	\parenth{
	\int_{\mathbb{T}^d}
	\abs{f(t)}^p
	\mathrm{d}t
	}^{1/p}
\end{equation}
is finite. For $p\geq 1$, \eqref{eq:lebesgue} is a norm; for $0<p<1$, it is a quasi-norm.
The $L_2$ Sobolev space of order 
$k\in\mathbb{Z}$ is $H_2^k(\mathbb{T}^d)$.
For $k\in \mathbb{N}$, $H_2^k(\mathbb{T}^d)$ is the collection of functions in $L_2(\T^d)$ with $k$ derivatives in $L_2(\mathbb{T}^d)$.

\subsection{Periodic Test Functions and Generalized Functions} 
\label{subsec:Sdot}

The standard Schwartz space of infinitely differentiable test functions is denoted as $\mathcal{S}(\mathbb{T}^d)$. The corresponding space of generalized functions is $\mathcal{S}^{\prime}(\mathbb{T}^d)$. These spaces are nuclear spaces \cite{Treves1967}.
Moreover, it is known that $\S(\R^d)$ and $\S'(\R^d)$ are respectively the intersection and the union of weighted Sobolev spaces~\cite{Simon2003distributions,Ito1984foundations} (or more generally weighted Besov spaces, see for instance~\cite{Kabanava2008tempered}).
This is easily inherited for the corresponding spaces on the torus; that is, 
\begin{align}
	\mathcal{S}(\mathbb{T}^d)
	&=
	\bigcap_{k\in\mathbb{Z}}
	H_2^k(\mathbb{T}^d), \\
	\mathcal{S}^{\prime}(\mathbb{T}^d)
	&=
	\bigcup_{k\in\mathbb{Z}}
	H_2^k(\mathbb{T}^d).
\end{align}

For our purpose, we consider the related spaces with mean zero.
Such generalized functions are well suited to wavelet approximation since wavelets also have mean zero.  In addition, this assumption simplifies the definition of stochastic processes.  For example, the derivative operator becomes a bijective mapping of homogeneous Sobolev and Besov spaces. Moreover, this assumption does not impact the generality of our results: indeed, the addition of a constant term does not affect the regularity of a function.    

By a $0$-mean generalized function $u \in \S'(\T^d)$, we mean one for which
$\langle u , 1 \rangle = 0$. 
We use a \emph{dot notation} to specify spaces of $0$-mean (generalized) functions. For instance, 
the $L_2$ Sobolev space of $0$-mean (generalized) functions of order
$k\in\mathbb{Z}$ is denoted as $\dot{H}_2^k(\mathbb{T}^d)$.

\begin{notation}
The space of infinitely differentiable test functions with mean zero is denoted as $\dot{\mathcal{S}}(\mathbb{T}^d)$. It has the projective limit topology
\begin{equation}
	\dot{\mathcal{S}}(\mathbb{T}^d)
	=
	\bigcap_{k\in\mathbb{Z}}
	\dot{H}_2^k(\mathbb{T}^d)
\end{equation}
The corresponding space of $0$-mean generalized functions (continuous linear functionals on $\dot{\mathcal{S}}(\mathbb{T}^d)$) with the weak-$^*$ topology is denoted as $\dot{\mathcal{S}}^{\prime}(\mathbb{T}^d)$.
\end{notation}

Note that $\dot{\mathcal{S}}(\T^d)$ has the same
nuclear-Fr\'echet-space structure as $\S(\T^d)$. Also, periodic generalized functions 
$f\in \dot{\mathcal{S}}^{\prime}(\mathbb{T}^d) $ are characterized by their Fourier coefficients $\innprod{f}{\ue^{2\pi \ui\innprod{\bm{m}}{\cdot} }}$, 
$\bm{m}\in \mathbb{Z}^d\backslash \{\bm{0}\}$. The zero term is excluded because the functions have mean $0$.


\subsection{Generalized Random Processes} 
\label{subsec:GRP}
	
The theory of generalized random processes was initiated independently by I. Gelfand and K. It\^o~\cite{Gelfand1955generalized,Ito1954distributions}, and further developed extensively in~\cite{GelVil4,Fernique1967processus,Ito1984foundations}. A recent introductory and concise presentation for \emph{tempered} generalized random processes, which can easily adapted to the case of (homogeneous) generalized functions on the torus, can be found in~\cite{Bierme2017generalized}. 

\begin{definition}
A \emph{generalized random process} on $\dot{\mathcal{S}}^{\prime}(\T^d)$ is  a collection of real random variables $(\langle s , \varphi \rangle)_{\varphi \in \dot{\mathcal{S}}(\T^d)}$ that satisfy the following properties:
\begin{enumerate}[i)]
\item 
\emph{Linearity.} 
$\langle s, \varphi  
+ \lambda \psi \rangle 
= 
\langle s,\varphi \rangle 
+ 
\lambda \langle s, \psi\rangle$ 
almost surely for every $\varphi, \psi \in \dot{\mathcal{S}}(\T^d)$.
\item 
\emph{Continuity.} 
$\langle s , \varphi_n\rangle  
\rightarrow
\langle s, \varphi \rangle$ in probability whenever 
$\varphi_n 
\rightarrow
\varphi$ in $\dot{\mathcal{S}}(\T^d)$. 
\end{enumerate}
\end{definition}

A generalized random process is therefore a continuous and linear functional from $\dot{S}(\T^d)$ to the space of random variables. Such an object is called a \emph{continuous linear random functional} in \cite{Ito1984foundations}.

\begin{definition}
The \emph{characteristic functional} 
$\CF_s:\dot{\mathcal{S}}(\T^d) \rightarrow \mathbb{C}$ of a generalized random process $s$ is defined as
\begin{equation}
	\CF_s(\varphi) 
	= 
	\mathbb{E}
	\left[ \mathrm{e}^{\mathrm{i} \langle s, 
	\varphi \rangle} \right].
\end{equation}
\end{definition}
A generalized random process $s$ specifies a  characteristic functional $\CF_s$ that is positive definite, continuous, and satisfies $\CF_s(0)= 1$.
Conversely, we can define a generalized random process by way of its characteristic functional. 
This is due to the structure of the nuclear Fr\'echet space 
$\dot{\mathcal{S}}^{\prime}(\mathbb{T}^d)$.
\begin{itemize}
\item The Minlos-Bochner theorem (\cite{GelVil4} or Theorem 2.4.1 and Theorem 2.4.3 of \cite{Ito1984foundations}) implies that a continuous, positive definite functional 
$\CF:\dot{\mathcal{S}}(\mathbb{T}^d)\rightarrow \mathbb{C}$ with $\CF(0) = 1$ is the Fourier transform of a probability measure $\mu$ on  
$\dot{\mathcal{S}}^{\prime}(\mathbb{T}^d)$ endowed with its cylindrical $\sigma$-field. 
\item A probability measure on $\dot{\mathcal{S}}^{\prime}(\mathbb{T}^d)$ is uniquely associated to a continuous linear random functional \cite[Chapter 2]{Ito1984foundations}.
\end{itemize}

We say that a functional $\CF$ is a \emph{characteristic functional} if it satisfies the properties of the Minlos-Bochner theorem.

\begin{example}
The functionals  
$\exp\parenth{-\norm{\varphi}_{L_\alpha(\mathbb{T}^d)}^\alpha}$ 
are positive definite and continuous on $\dot{\mathcal{S}}(\mathbb{T}^d)$
for $0<\alpha \leq 2$. Accordingly, they are  characteristic functionals that specify  
some corresponding
generalized random processes on $\dot{\mathcal{S}}^{\prime}(\mathbb{T}^d)$, namely, \sas \ periodic L\'evy white noises. 
\end{example}

\subsection{Homogeneous Besov Spaces} 
\label{subsec:Besov}	
	
The homogeneous Besov spaces $\dot{B}_{p,q}^{\tau}(\mathbb{T}^d)$ are specified by two primary parameters ($p$ and $\tau$) and one secondary 
parameter ($q$).
The parameter $p\in(0,\infty]$ plays a role that is similar to the index defining the Lebesgue spaces $L_p(\mathbb{T}^d)$, while $\tau\in\mathbb{R}$ indicates smoothness
in the sense of order of differentiability.
Therefore, roughly speaking, for $\tau\in\mathbb{N}$, a function in $\dot{B}_{p,q}^{\tau}(\mathbb{T}^d)$ has $\tau$ derivatives in $L_p(\mathbb{T}^d)$. In Figure \ref{fig:besov}, we provide a structural diagram that represents the collection of Besov spaces. Our interest in these spaces lies in the following facts:

\begin{enumerate}[i)]
\item 
Wavelets form unconditional bases for Besov spaces;
\item 
The mapping that takes a function to its wavelet coefficients is an isomorphism between Besov function spaces and Besov sequence spaces; 
\item 
The $\nn$-term approximation characterizes Besov sequence spaces. Given a sequence with a known rate of $\nn$-term approximation, we can specify which Besov sequence spaces it is in.    
\end{enumerate}

\begin{figure}[ht] 
\centering

\begin{tikzpicture}[x=4cm,y=3cm,scale=1]
\draw[very thick, ->] (-1,0)--(0.5,0) node[circle,right] {$\frac{1}{p}$} ;
\draw[very thick, <->] (-1,-0.5)--(-1,0.5) node[circle,above] {$\tau$} ;

\draw[ thick,color=black] (-1+0.05,1/4) -- (-1-0.05,1/4)  node[black,left] { $\tau_0$};

\draw[ thick,color=black] (-0.5,0.05) -- (-0.5,-0.05)  node[black,below] { $1/p_0$};

\draw[]
(-0.5,0.25) 
node[black, circle, minimum width=4pt, fill,fill opacity=1,  inner sep=0pt]{}
node[black, above] {$\dot{B}_{p_0,q}^{\tau_0}(\mathbb{T}^d) $};
\end{tikzpicture}

\caption{
Function-space diagram. A point $(1/p_0,\tau_0)$ represents all Besov spaces  $\dot{B}_{p_0,q}^{\tau_0}(\mathbb{T}^d)$ for $0<q \leq \infty$. 
}
\label{fig:besov}
\end{figure}
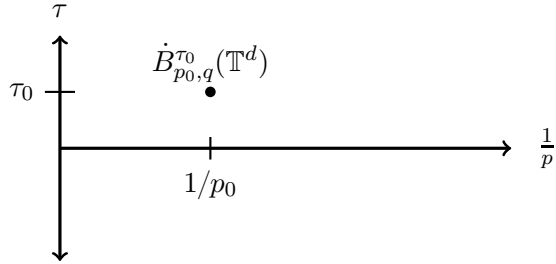

The classical definition of Besov spaces is taken from \cite[Definition 1.27]{triebel08} and repeated in Definition \ref{prop:def_besov}. The idea is to decompose a function $f$ by grouping dyadic frequency bands, using a partition of unity in the Fourier domain.

\begin{definition}
Let $\widehat{\upsilon} \in \mathcal{S}(\R^d)$ generate a hierarchical partition of unity outside the  ball of radius $1/2$ centered at the origin. Specifically, 
\begin{enumerate}[i)]
\item 
${\displaystyle  
\widehat{\upsilon}(\bm{\omega}) = 0}$ 
if  
$\abs{\bm{\omega}}\leq 1/2$  
or 
$\abs{\bm{\omega}}\geq 2$
\item 
${\displaystyle  
\widehat{\upsilon}(\bm{\omega}) > 0}$ 
if  
$1/2< \abs{\bm{\omega}} < 2$ 
\item 
${\displaystyle  
\sum_{j=0}^{\infty} \widehat{\upsilon}(2^{-j}\bm{\omega}) = 1 }$ 
if  
$1 \leq \abs{\bm{\omega}}$. 
\end{enumerate}
\end{definition}
Note that the decomposition of a function $f$ into the components 
\begin{equation}
\sum_{\bm{m}\in\mathbb{Z}^d\backslash \set{\bm{0}}  }
\widehat{f}(\bm{m})
\widehat{\upsilon}(2^{-j} \bm{m})
\ue^{2\pi \ui \innprod{\bm{m}}{\cdot}}
\end{equation}
is intimately related to the concept of a wavelet decomposition, with 
$\upsilon$ playing the role of a mother wavelet.  

\begin{definition}\label{prop:def_besov}
Suppose $0<p,q \leq \infty$ and $\tau \in\mathbb{R}$. A generalized function $f\in \dot{\mathcal{S}}^{\prime}(\mathbb{T}^d)$ with Fourier coefficients $\widehat{f}(\bm{m})$ is in $\dot{B}_{p,q}^{\tau}(\mathbb{T}^d)$ if the quantity 
\begin{equation} \label{eq:besovnorm}
	\parenth{
	\sum_{j=0}^{\infty}
	2^{j\tau q}
	\norm{
	\sum_{\bm{m}\in\mathbb{Z}^d\backslash \set{\bm{0}}  }
	\widehat{f}(\bm{m})
	\widehat{\upsilon}(2^{-j} \bm{m})
	\ue^{2\pi \ui \innprod{\bm{m}}{\cdot}}
	}_{L_p(\mathbb{T}^d)}^q
	}^{1/q}
\end{equation}
is finite. For $q=\infty$, the norm must be suitably modified.
\end{definition}

Besov spaces are Banach spaces for the norm \eqref{eq:besovnorm} when $p$ and $q \geq 1$. For $p$ or $q<1$, \eqref{eq:besovnorm} is a quasi-norm and the Besov spaces are quasi-Banach spaces. 
The validity of the embeddings between Besov spaces is governed by Proposition \ref{prop:embeddings} \cite{triebel08}.

\begin{proposition} \label{prop:embeddings}
	Let $(- \infty) < \tau_0, \tau_1 <  \infty$ and $0<p_0, p_1 < \infty$. Then, the topological embedding $\dot{B}_{p_1,p_1}^{\tau_1} (\T^d) \subseteq \dot{B}_{p_0,p_0}^{\tau_0} (\T^d)$ is valid in the following cases:
	\begin{itemize}
		\item if $p_0 \leq p_1$ and $\tau_0 < \tau_1$; or
		\item if $p_0 > p_1$ and 
		$\tau_0 
		< 
		\tau_1 
		+ 
		d\left( \frac{1}{p_0} - \frac{1}{p_1}\right)$.
	\end{itemize}
\end{proposition}


\section{Generalized L\'evy Processes on the Torus}
\label{sec:sparse}

The main objects of study in this paper are generalized L\'evy processes $s$ that are solutions of the stochastic differential equation $\Lop s = w$, with $w$ a L\'evy white noise. If $w$ has no Gaussian part, then $s$ is a sparse process. In this section, we introduce the family of L\'evy white noises $w$ and specify the class of considered operators $\Lop$. 

\subsection{L\'evy White Noises and Their Besov Regularity} \label{subsec:noises}
L\'evy white noises have been introduced as generalized random processes on the Schwartz space $\mathcal{D}'(\R^d)$ of generalized functions in \cite{GelVil4}.
They are commonly defined through their characteristic functional, relying on the Minlos-Bochner theorem (see Section \ref{subsec:GRP}). 

Given a probability space $\Omega$, a random variable $X:\Omega\rightarrow \R$ is infinitely divisible if it equal in law to  a sum of $N$ i.i.d. random variables for any $N \geq 1$ \cite{Sato94}. 
L\'evy white noises are intimately connected with infinite-divisible laws, the finite-dimensional marginals of those processes being all infinitely divisible.
The characteristic function of an infinitely divisible random variable $X$ can be written as
\begin{equation}
	\Phi_X(\xi) = \exp ( \lexp ( \xi ) )
\end{equation}
with $\lexp$ a continuous and conditionally positive-definite function with value $0$ at $\xi = 0$ \cite{Sato94,Unser2014sparse}. 
The continuous log-characteristic function of an infinitely divisible random variable is called a \emph{L\'evy exponent} or a \emph{characteristic exponent}. 
Definition \ref{def:grp_cf} is the adaptation of the usual definition of L\'evy white noise \cite{GelVil4} to the nuclear space $\dot{S}'(\T^d)$. 

\begin{definition}\label{def:grp_cf}
A L\'evy exponent $\lexp$ specifies a generalized random process $w$ in $\dot{S}'(\T^d)$ with characteristic functional 
\begin{equation}
	\CF_w(\varphi) 
	= 
	\exp \left( \int_{\T^d} \lexp (\varphi (\bm{x}) ) \drm \bm{x} \right),
\end{equation}
where $\varphi \in \dot{S}(\T^d)$. 
We call such a process $w$ a \emph{L\'evy white noise}.
\end{definition}

Note that this defines a  characteristic functional, as seen in Section \ref{subsec:GRP} and \cite[Chapter 3]{GelVil4}, where the arguments are easily adapted from $\D'(\R^d)$ to $\dot{S}' (\T^d)$. 

The L\'evy-Khintchine theorem \cite{Sato94} ensures that a L\'evy exponent $\lexp$   can be decomposed as 
\begin{equation}
	\lexp(\xi) = \mathrm{i} \mu \xi - \frac{\sigma^2\xi^2}{2} + \int_{\R} (\mathrm{e}^{\mathrm{i} \xi t } - 1 - \mathrm{i}\xi t 1_{\abs{t} \leq 1}) \nu( \drm t)
\end{equation}
where $\mu \in \R$, $\sigma^2\geq 0$, and $\nu$ is a L\'evy measure, which means that it is a measure such that $\int_{\R} \inf( 1, t^2 ) \nu(\drm t) < \infty$ and $\nu \{0\} = 0$. We say that the white noise is \emph{Gaussian} if $\mu$ and $\nu$ are $0$. In that case, 
\begin{equation}
	\CF_w(\varphi) 
	= 
	\exp
	\parenth{ - \frac{\sigma^2 \lVert \varphi \rVert^2_2}{2}}
\end{equation}
and we recover the usual Gaussian white noise. If $\sigma^2 = 0$ (\emph{i.e.}, if $w$ has no Gaussian part), then  we say that $w$ is \emph{sparse} \cite{Unser2014sparse}. 

We shall deduce the regularity of a solution of \eqref{eq:ls=w} from the regularity of the underlying white noise $w$. This regularity can be computed in terms of the  Blumenthal-Getoor indices of this noise.  

\begin{definition}
The \emph{Blumenthal-Getoor indices} of a 
L\'evy white noise with L\'evy exponent $f$ are defined as
\begin{align} \label{eq:BGindex}
	\beta 
	&= 
	\inf 
	\left\{ 
	p\in [0,2], \  
	\limsup_{\abs{\xi} \rightarrow \infty} 
	\frac{\abs{\lexp(\xi)}}{\abs{\xi}^p}
	= 0
	\right\}.\\
	\beta'
	&= 
	\inf 
	\left\{ 
	p\in [0,2], \  
	\liminf_{\abs{\xi} \rightarrow \infty} 
	\frac{\abs{\lexp(\xi)}}{\abs{\xi}^p}
	= 0
	\right\}.
\end{align}	
\end{definition}

The index $\beta$ was introduced by R. Blumenthal and R.K. Getoor in \cite{Blumenthal1961sample} to characterize the behavior of a L\'evy process around the origin. 
Since then, it has been used to characterize many local properties of such processes, including the spectrum of singularities in multifractal analysis~\cite{durand2012multifractal,Jaffard1999multifractal} and local self-similarity~\cite{Fageot2016scaling}. 
More importantly for our use, $\beta$ allows us to characterize  L\'evy process~\cite{schilling1997Feller,Bottcher2013levy} and noises \cite{FUW2015Besov,Fageot2016multidimensional,Aziznejad2018sharp} in terms of Besov spaces
\footnote{The cited works deal with general L\'evy type processes that do not necessarily have stationary increments.}.

While equivalent for all of the classical L\'evy noises (including Gaussian, 
S$\alpha$S, and compound Poisson,) the index $\beta'$ differs in general from $\beta$. 
A counterexample where $\beta' < \beta$ can be found in~\cite{Farkas2001function}. 
Importantly, $\beta'$ plays a crucial role in the Besov space characterization of L\'evy noises; it is used to determine an upper bound on the Besov regularity~\cite{Aziznejad2018sharp}. 

In Theorem \ref{theo:noisebesov}, we summarize the results obtained in our previous works~\cite{FUW2015Besov,Fageot2016multidimensional,Aziznejad2018sharp}.  The two latter references deal with processes defined over $\R^d$; they can easily be adapted. For a direct construction and study of L\'evy white noises on the torus, see the former~\cite{FUW2015Besov}. Here, \emph{a.s.} means \emph{almost surely}.
	
\begin{theorem}
\label{theo:noisebesov}
We consider a L\'evy white noise $w$ with Blumenthal-Getoor indices $0 \leq \beta' \leq \beta \leq 2$. We fix $0<p,q \leq \infty$ and $\tau \in \R$.
\begin{itemize}
\item If $w$ is Gaussian, then
	\begin{align}
		w \in \dot{B}_{p,q}^{\tau}(\T^d) \text{ a.s. if } & \tau < - d / 2, \text{ and} \label{eq:gausspositive}\\
		w \notin \dot{B}_{p,q}^{\tau}(\T^d) \text{ a.s. if } & \tau > - d / 2. \label{eq:gaussnegative}
	\end{align}
\item If $w$ is non Gaussian, then 
	\begin{align}
		w \in \dot{B}_{p,q}^{\tau}(\T^d) \text{ a.s. if } & \tau <d \left( \frac{1}{\max(p,\beta)} - 1\right), \text{ and} \label{eq:levypositive}\\
		w \notin \dot{B}_{p,q}^{\tau}(\T^d) \text{ a.s. if } & \tau >d \left( \frac{1}{\max(p,\beta')} - 1\right). \label{eq:levynegative}
	\end{align}
\end{itemize}
\end{theorem}

The Gaussian noise differs from the non-Gaussian L\'evy noises in one main way: for $p \geq 2$, the critical value for the non-Gaussian case is $d( 1 / p - 1)$, while it is $-d/2$ for the Gaussian case. The Besov regularity of the Gaussian noise on the torus has been studied in detail in \cite{veraar2010regularity} using Fourier series techniques. We have re-obtain similar results in~\cite[Section 3]{Aziznejad2018sharp} with wavelet-based methods. These two references allow us to deduce \eqref{eq:gausspositive} and \eqref{eq:gaussnegative}. The positive result \eqref{eq:levypositive} can be found in \cite{FUW2015Besov} for S$\alpha$S noise and in \cite{Fageot2016multidimensional} for the general case. The negative result \eqref{eq:levynegative} is taken from \cite{Aziznejad2018sharp}. 
	\subsection{Differential Operators of Order $\gamma$} \label{subsec:operators}

We shall consider the class of differential operators that reduce the Besov regularity of a function by some (possibly fractional) order $\gamma>0$. 
Importantly, 
since we are interested in the regularity properties of the solutions of the differential equation 
$\Lop s=w$, 
we focus on those operators that are continuous bijections from $\dot{B}^{\tau + \gamma}_{p,q}(\T^d)$ to $\dot{B}^{\tau}_{p,q}(\T^d)$. 
In this case, 
the regularity properties of the white noise allow us to deduce the regularity properties of the process 
$s$. 

\begin{definition}
We restrict our attention to the Fourier-multiplier operators 
$\Lop:
\dot{\mathcal{S}}^{\prime}(\T^d) 
\rightarrow 
\dot{\mathcal{S}}^{\prime}(\T^d)$, 
specified by a symbol 
$\widehat{L}:\R^d \rightarrow \C$, where
\begin{equation}
	\Lop f 
	=
	\mathcal{F}^{-1} 
	\set{ 
	\widehat{f} 
	\left.
	\widehat{L}
	\right|_{\Z^d}  
	}.
\end{equation}
A bijective operator of this form is said to be \emph{admissible}.
\end{definition}

\begin{remark}
Let us point out that an admissible operator $\Lop$ can be initially defined as a mapping from  $\dot{\mathcal{S}}(\T^d)$ to $\dot{\mathcal{S}}(\T^d)$ since there is a natural extension by duality to the space of generalized functions. Also, notice that we require the symbol $\widehat{L}$ to be defined on the continuous domain $\R^d$, even though the action of $\Lop$ is determined by the values of $\widehat{L}$ on $\Z^d$.  Our reason will be clear in Theorem \ref{th:perturb}, where we use the noninteger values on $\R^d$.
\end{remark}

\begin{definition}
An admissible operator $\Lop$ is said to be 
\emph{$\gamma$-admissible} for $\gamma \in \R$ if 
$\Lop: 
\dot{B}^{\tau + \gamma}_{p,q}(\T^d) 
\rightarrow 
\dot{B}^{\tau}_{p,q}(\T^d)$
is a continuous bijection and $\Lop^{-1}$ is continuous for every $0<p,q \leq \infty$ and $\tau \in \R$.
\end{definition}

The fractional Laplacian of order $\gamma>0$ $(-\Delta)^{\gamma/2} $ is the canonical example of a  $\gamma$-admissible operator. For a generalized function 
\begin{equation}
f(\bm{x})
=
\sum_{\bm{m}\in\mathbb{Z}^d\backslash \set{\bm{0}}  }
\widehat{f}(\bm{m})
\ue^{2\pi \ui \innprod{\bm{m}}{\cdot}},
\end{equation}
the fractional Laplacian of $f$ is
\begin{equation}\label{eq:fac_lap_f}
(-\Delta)^{\gamma/2} f(\bm{x})
=
\sum_{\bm{m}\in\mathbb{Z}^d\backslash \set{\bm{0}}  }
\widehat{f}(\bm{m})
\abs{\bm{m}}^\gamma
\ue^{2\pi \ui \innprod{\bm{m}}{\cdot}}.
\end{equation}
Note that this operator maps $0$-mean generalized functions to $0$-mean generalized functions.

Moreover, perturbations of the fractional Laplacian are also $\gamma$-admissible. The next few results make this statement precise. The idea is the following: An operator $\Lop$ is $\gamma$-admissible if and only if  
$(-\Delta)^{\gamma/2}\Lop^{-1} $  and $(-\Delta)^{-\gamma/2}\Lop $ are automorphisms on Besov spaces.

\begin{proposition}
The fractional Laplacian  $(-\Delta)^{\gamma/2} $ is a  $\gamma$-admissible operator.
\end{proposition}
\begin{proof}

Consider the norm of Definition \ref{prop:def_besov} applied to the generalized function $\eqref{eq:fac_lap_f}$. To verify the result, we must bound the term
\begin{equation}\label{eq:frac_lap_lp}
\norm{
	\sum_{\bm{m}\in\mathbb{Z}^d\backslash \set{\bm{0}}  }
	\abs{\bm{m}}^{\gamma}
	\widehat{f}(\bm{m})
	\widehat{\upsilon}(2^{-j} \bm{m})
	\ue^{2\pi \ui \innprod{\bm{m}}{\cdot}}
	}_{L_p(\mathbb{T}^d)}
\end{equation}
by a constant multiple of 
\begin{equation}
2^{j \gamma}
\norm{
	\sum_{\bm{m}\in\mathbb{Z}^d\backslash \set{\bm{0}}  }
	\widehat{f}(\bm{m})
	\widehat{\upsilon}(2^{-j} \bm{m})
	\ue^{2\pi \ui \innprod{\bm{m}}{\cdot}}
	}_{L_p(\mathbb{T}^d)}
\end{equation}
for arbitrary $j$.

Let us define 
$\bar{v} = \widehat{v}(2 \cdot)+\widehat{v} + \widehat{v}(2^{-1}\cdot) $, and notice that 
\begin{equation}
	\abs{\bm{m}}^{\gamma}
	\widehat{f}(\bm{m})
	\widehat{\upsilon}(2^{-j} \bm{m})
	\ue^{2\pi \ui \innprod{\bm{m}}{\cdot}}
	=
	\abs{\bm{m}}^{\gamma}
	\bar{v}(2^{-j}\bm{m})
	\widehat{f}(\bm{m})
	\widehat{\upsilon}(2^{-j} \bm{m})
	\ue^{2\pi \ui \innprod{\bm{m}}{\cdot}}.
\end{equation}
Then we can apply Theorem 3.3.4 of \cite{schmeisser87} in the following way. The multiplier is $\abs{\bm{m}}^{\gamma}\bar{v}(2^{-j}\bm{m})$ and the dilation factor is $2^{j+2}$. Hence there is a constant $C$ for which \eqref{eq:frac_lap_lp} is bounded by
\begin{equation}\label{eq:frac_lap_two_norms}
C
\norm{
\abs{2^{j+2}\cdot}^{\gamma}
\bar{v}(2^{-j}2^{j+2}\cdot)
}_{H_2^k(\mathbb{R}^d)}
\norm{
	\sum_{\bm{m}\in\mathbb{Z}^d\backslash \set{\bm{0}}  }
	\widehat{f}(\bm{m})
	\widehat{\upsilon}(2^{-j} \bm{m})
	\ue^{2\pi \ui \innprod{\bm{m}}{\cdot}}
	}_{L_p(\mathbb{T}^d)}.
\end{equation}
for some sufficiently large $k$ depending on $p$ and $d$.
The first term of \eqref{eq:frac_lap_two_norms} is bounded using homogeneity
\begin{equation}
\norm{
\abs{2^{j+2}\cdot}^{\gamma}
\bar{v}(2^{-j}2^{j+2}\cdot)
}_{H_2^k(\mathbb{R}^d)}
=
2^{\gamma (j+2)}
\norm{
\abs{\cdot}^{\gamma}
\bar{v}(4\cdot)
}_{H_2^k(\mathbb{R}^d)}.
\end{equation}
This last expression is finite due to the smoothness of $\widehat{v}$.
\end{proof}

\begin{theorem}\label{th:perturb}
Let $\Lop$ be an admissible operator with symbol $\widehat{L}$. 
For $\gamma>0$, define $m(\bm{\omega})= \abs{\bm{\omega}}^{-\gamma} \widehat{L}(\bm{\omega})$. 
Also,
let $\zeta $ be any function in $ \mathcal{S}(\R^d)$ satisfying
\begin{align}
	0&\leq \zeta(\bm{x}) \leq 1, \quad
	\zeta(\bm{x}) =
	\begin{cases}
	0, & \abs{\bm{x}}\leq 1/4 \\
	1, & 1/2 \leq \abs{\bm{x}}\leq 2 \\
	0, & \abs{\bm{x}}\geq 4.
	\end{cases}
\end{align}
If 
\begin{equation}
	\sup_{j\in \Z_{\geq 0}} 
	\parenth{
	\norm{\zeta m(2^{j}\cdot)}_{H_{2}^{\tau}(\R^d)}	
	+
	\norm{\zeta m(2^{j}\cdot)^{-1}}_{H_{2}^{\tau}(\R^d)}	
	}
	<
	\infty
	\quad
	\text{for all}
	\quad 
	\tau>0,
\end{equation}
then $\Lop$ is $\gamma$-admissible. 
\end{theorem}

\begin{proof}
This follows from a sufficient condition for Fourier multipliers on Besov spaces, Theorem 3.6.3 of \cite{schmeisser87}. To summarize, 
if $0<p<\infty$, $0<q<\infty$, $(-\infty)<\tau<\infty$, and
\begin{equation}
	\tau > d \parenth{ \frac{1}{\min\parenth{1,p}}-\frac{1}{2}},
\end{equation}
then there exists a positive constant $C$ such that
\begin{equation}
	\norm{
	\sum_{\bm{k}\in\Z^d} 
	m(\bm{k}) 
	\widehat{f}(\bm{k})
	\ue^{2\pi \ui \innprod{\bm{k}}{ \cdot}} 
	}_{\dot{B}_{p,q}^{\tau}(\T^d)}
	\leq
	C
	\parenth{
	\sup_{j\in \Z_{\geq 0}} 
	\norm{
	\zeta m(2^{j}\cdot)
	}_{H_{2}^{\tau}(\R^d)}	
	}
	\norm{f}_{\dot{B}_{p,q}^{\tau}(\T^d)} 
\end{equation}
holds for all functions $m\in L_{\infty}(\mathbb{R}^d)$ and all  $f \in \dot{B}_{p,q}^{\tau}(\T^d)$. 
\end{proof}

\begin{example}

$\gamma$-Admissible Operators
\begin{enumerate}[i)]
\item 
The derivative $\Der$ is $1$-admissible. 
\item 
The differential operators  
$\Der^{\gamma} + a_{\gamma-1} \Der^{\gamma-1} + \cdots + a_0 \mathrm{Id}$ with non-vanishing symbols are $\gamma$-admissible for 
$\gamma \in \mathbb{Z}_{\geq0}$
\item 
The fractional derivative $\Der^\gamma$ is $\gamma$-admissible for any $\gamma > 0$.
\item 
The fractional Laplacian $\FL$ is $\gamma$-admissible for any $\gamma > 0$.
\item 
The Mat\'ern operator $(\mathrm{Id} - \Delta)^{\gamma /2}$ is $\gamma$-admissible for any $\gamma > 0$.
\end{enumerate}
\end{example}

\subsection{Besov Regularity of Generalized L\'evy Processes} \label{subsec:Besovsparseprocesses}

We are now in a position to deduce the Besov regularity of generalized L\'evy processes. Corollary \ref{coro:besovregus} directly follows from the regularity of a white noise process (Theorem \ref{theo:noisebesov}) and from the definition of a $\gamma$-admissible operator.
	
\begin{corollary} \label{coro:besovregus}
We consider a stochastic process $s = \Lop^{-1} w$, where $w$ is a L\'evy white noise  with Blumenthal-Getoor indices $0\leq \beta' \leq \beta \leq 2$ and $\Lop$ is a $\gamma$-admissible operator for some $\gamma \geq 0$.
\begin{itemize}
\item If $w$ is a Gaussian white noise, then
	\begin{align}
		s \in \dot{B}_{p,q}^{\tau}(\T^d) \text{ a.s. if } & \tau < \gamma  - d / 2, \text{ and} \label{eq:GaussBesov}\\
		s \notin \dot{B}_{p,q}^{\tau}(\T^d) \text{ a.s. if } & \tau > \gamma  - d / 2. \label{eq:GaussBesovnegative}
	\end{align}
\item If $w$ is non Gaussian, then 
	\begin{align}
		s \in \dot{B}_{p,q}^{\tau}(\T^d) \text{ a.s. if } & \tau <\gamma + d \left( \frac{1}{\max(p,\beta)} - 1\right), \text{ and} \label{eq:sparseBesov}\\
		s \notin \dot{B}_{p,q}^{\tau}(\T^d) \text{ a.s. if } & \tau > \gamma + d \left( \frac{1}{\max(p,\beta')} - 1\right). \label{eq:sparseBesovnegative}
	\end{align}
\end{itemize}
\end{corollary}



\section{The $\nn$-Term Approximation and the Compressibility of Generalized L\'evy Processes} 
\label{sec:Nterm}

In Section \ref{sec:sparse}, 
we found an upper bound for the Besov regularity of a generalized L\'evy process. 
We are now interested in using that result to determine its $\nn$-term approximation with respect to Daubechies wavelet bases. 
We begin by defining the wavelets. 
Then, 
we recall the wavelet characterization of Besov spaces.  

\subsection{Periodic Daubechies Wavelets}
\label{subsec:daubechies}
Here, 
mainly following \cite{triebel08},  
we introduce the family of Daubechies wavelets on the $d$-dimensional torus. 
We also give a wavelet-based characterization of homogeneous Besov spaces. 

Periodizing the compactly supported Daubechies wavelets \cite{daubechies92} results in the  orthonormal basis of $L_2(\mathbb{T}^d)$ given by
\begin{equation}
	\setb{\Psi_{G,\bm{m}}^{j,k}}
	{
	j\in \mathbb{Z}_{\geq 0}, 
	G \in G^j, 
	\bm{m}\in \mathbb{P}_j^d 
	}, 
\end{equation}
where
\begin{equation}
	\Psi_{G,\bm{m}}^{j,k}
	= 
	2^{jd/2}\Psi_{G,0}^{0,k}(2^j\cdot-\bm{m}).
\end{equation}
The index $j\in\mathbb{Z}_{\geq 0}$ corresponds to a scaling parameter, and $G$ is used to denote gender.
The coarsest scale is $j=0$, which includes the scaling functions, so 
$G^0$ has $2^d$ elements and $G^j$ has $(2^d-1)$ elements for  $j>0$. The parameter $k$ denotes the smoothness of the wavelet and  determines its support. For $k>0$, the classical Daubechies mother wavelet on the real line has support greater than one. Here, we require the support of the wavelets to be a subset of the unit cube. Consequently, the coarsest scale is scaled by $2^{\waveexp}$, where the parameter ${\waveexp}\in\mathbb{N}$ ensures that this condition is satisfied. For the remainder of this paper, we set ${\waveexp}$ (as a function of $k$) to be the smallest integer that guarantees this  condition on the support. The wavelet translates are indexed by $\bm{m}$, and the set of translations at scale $j$ is
\begin{equation}
	\mathbb{P}_j^d 
	=
	\setb{\bm{m}\in\mathbb{Z}^d}
	{0\leq m_r < 2^{j+{\waveexp}}, r = 1,\dots,d}.
\end{equation}

More details on the periodization of wavelet bases  can be found in \cite[Section 1.3]{triebel08}. 
In fact, we nearly follow the notation of that book, except that our $\Psi_{G,\bm{m}}^{j,k}$ corresponds to $\Psi_{G,\bm{m}}^{j,\text{per}}$ of \cite[Proposition 1.34]{triebel08}. 

\begin{definition}
The notation $\Psi_{G_0,\bm{0}}^{0,k}$ denotes a Daubechies wavelet  in $C^k(\mathbb{T}^d)$. Furthermore, the Lebesgue measure of the support of
$\Psi_{G_0,\bm{0}}^{0,k}$ is less than one.
\end{definition}

The wavelet decomposition of $f \in L_2(\T^d)$ is
\begin{equation}
	f
	=
	\sum_{j,G,\bm{m}} 	
	\lambda_{\bm{m}}^{j,G}\,
	2^{-(j+{\waveexp})d/2}
	\Psi_{G,\bm{m}}^{j,k} ,
\end{equation}
where the coefficients are computed as
\begin{equation}\label{eq:wave_coef}
	\lambda_{\bm{m}}^{j,G}
	=
	\innprod{f}{2^{(j+{\waveexp})d/2}
	\Psi_{G,\bm{m}}^{j,k}}.
\end{equation}

We can now state the characterization of the homogeneous Besov spaces 
$\dot{B}_{p,q}^{\tau}(\mathbb{T}^d)$ that will be used in the remainder  of this paper. Essentially, it says that the Besov regularity of a function is characterized by a norm on its wavelet coefficients.

\begin{definition}[Definition 1.3.2, \cite{triebel08}]
Let $\tau\in \mathbb{R}$ and  $0<p,q < \infty$. The Besov sequence space $b_{p,q}^{\tau}$ is the collection of sequences
\begin{equation}
	\lambda
	=
	\setb{\lambda_{\bm{m}}^{j,G}}
	{j\in\mathbb{Z}_{\geq 0}, G \in G^j, \bm{m} \in \mathbb{P}_j}
\end{equation} 
indexed as the periodized Daubechies wavelets, with finite (quasi-)norm
\begin{equation}
	\norm{\lambda}_{b_{p,q}^{\tau}}
	:=
	\parenth{
	\sum_{j=0}^{\infty}
	2^{j(\tau-d/p)q}
	\sum_{G\in G^j}
	\parenth{
	\sum_{\bm{m}\in \mathbb{P}_j^n}
	\abs{\lambda_{\bm{m}}^{j,G}}^p
	}^{q/p}
	}^{1/q}.
\end{equation}
If $p=\infty$ or $q=\infty$, there is an analogous definition. 
\end{definition}
In the special case $p=q$, we have
\begin{equation}
	\norm{\lambda}_{b_{p,p}^{\tau}}
	:=
	\parenth{
	\sum_{j,G,\bm{m}}
	\abs{2^{j(\tau-d/p)} \lambda_{\bm{m}}^{j,G}}^p
	}^{1/p},
\end{equation}
which is a weighted $\ell_p$ space.

\begin{proposition}\label{prop:besov_char}
Suppose $f\in \dot{B}_{2,2}^{-k}(\mathbb{T}^d)$ (which is equivalent to the Sobolev space $\dot{H}_2^{-k}(\mathbb{T}^d)$) for some $k\in\mathbb{N}$, and let $0<p,q\leq \infty$ and 
$\tau\in\mathbb{R}$ such that 
$k 
> 
\max\parenth{\tau, \sigma_p-\tau}$, 
where 
$\sigma_p = d(1/p-1)_{+}$. 
Then, $f\in \dot{B}_{p,q}^{\tau}(\mathbb{T}^d)$ if and only if the wavelet coefficients	
\begin{equation}
	\innprod{f}{2^{(j+{\waveexp})d/2}
	\Psi_{G,\bm{m}}^{j,k+1}}
\end{equation}
are in the Besov sequence space $b_{p,q}^{\tau}$.
\end{proposition}  
\begin{proof}
This follows from Theorem 1.36 of \cite{schmeisser87}.

\end{proof}

\subsection{Besov Spaces and $\nn$-Term Approximations} \label{subsec:BesovNterm}

Summarizing the results of Section \ref{subsec:daubechies}, we have the following: If $f$ is in the Besov space $B_{p,q}^{\tau}(\T^d)$, then we can choose $k$
large enough so that 
\begin{equation}\label{eq:wave_expansion}
	f
	=
	\sum_{j,G,\bm{m}} 	
	\lambda_{\bm{m}}^{j,G} \,
	2^{-(j+{\waveexp})d/2}
	\Psi_{G,\bm{m}}^{j,k}, 
\end{equation}
where the coefficients $\lambda_{\bm{m}}^{j,G}$ are computed by
\eqref{eq:wave_coef}.
We are now going to determine the error in approximating $f$ by truncating the sum \eqref{eq:wave_expansion}. In order to accomplish this, we introduce the notation 
$\mathcal{I} = \mathbb{Z}_{\geq 0 }\times G^j \times \mathbb{P}_j^d $ to represent the collection of triples $(j,G,\bm{m})$. 
\begin{definition}
For $n\geq 1$, an $\nn$-term approximation to a function $f$ is a finite sum of the form
\begin{equation}
	\sum_{(j,G,\bm{m})\in \mathcal{I}^{\prime}} 	
	\lambda_{\bm{m}}^{j,G} \,
	2^{-(j+{\waveexp})d/2}
	\Psi_{G,\bm{m}}^{j,k},
\end{equation}
where $\mathcal{I}^{\prime} \subset \mathcal{I}$ and 
$\# \mathcal{I}^{\prime} = {\nn}$.

\end{definition}

\begin{definition}
Let $\Sigma_{\nn,p,\tau}(f)$ be a best $\nn$-term approximation to $f$ in 
$\dot{B}_{p,p}^{\tau}(\T^d)$, in other words, an $\nn$-term approximation that minimizes the approximation error in $\dot{B}_{p,p}^{\tau}(\T^d)$. Also, let $\sigma_{\nn,p,\tau}(f)$ denote the error of the approximation, with
\begin{equation}
	\sigma_{\nn,p,\tau}(f)
	=
	\norm{f-\Sigma_{\nn,p,\tau}(f)}_{\dot{B}_{p,p}^{\tau}(\T^d)}.
\end{equation}
\end{definition}

\begin{theorem} \label{theo:NtermBesov}
Suppose $0<p_0<\infty$ and $\tau_0 \in \mathbb{R}$. 
\begin{enumerate}[i)]
\item 	
If $f\in \dot{B}_{p_1,p_1}^{\tau_0 + \Delta\tau}(\T^d)$  for some
$\Delta\tau>0$ and $p_1$ satisfies \eqref{eq:p0p1gamma}, then there is a constant $C>0$ such that
\begin{equation}
	\sigma_{\nn,p_0,\tau_0}(f)
	\leq
	C \nn^{-\Delta \tau/d}
	\norm{f}_{\dot{B}_{p_1,p_1}^{\tau_0+\Delta\tau}(\T^d)}.
\end{equation}
\item 
If there are constants $C,\Delta\tau,\epsilon>0$  such that 
\begin{equation}
	\sigma_{\nn,p_0,\tau_0}(f)
	\leq
	C {\nn}^{-\Delta\tau/d-\epsilon},
\end{equation}
then $f \in \dot{B}_{p_1,p_1}^{\tau_0+\Delta\tau}(\T^d)$, where
\begin{equation}\label{eq:p0p1gamma}
	\frac{1}{p_1} = \frac{\Delta\tau}{d} +\frac{1}{p_0}.
\end{equation}
\end{enumerate}
\end{theorem}
\begin{proof}
This proof uses Corollary 6.2 of \cite{garrigos04}, which characterizes $\nn$-term approximation spaces as Besov spaces.  In particular, 
\begin{equation}\label{eq:besov_char}
	b_{p_1,p_1}^{\tau_0+\Delta\tau}
	=
	A_{p_1}^{\Delta\tau/d} (b_{p_0,p_0}^{\tau_0}), 
\end{equation}
where $b_{p_1,p_1}^{\tau_0+\Delta\tau}$ is a Besov sequence space, and 
$A_{p_1}^{\Delta\tau/d} (b_{p_0,p_0}^{\tau_0})$ is an approximation space 
with error measured in $b_{p_0,p_0}^{\tau_0}$. Essentially, 
$A_{p_1}^{\Delta\tau/d} (b_{p_0,p_0}^{\tau_0})$ is the collection of sequences $f$ for which the sequence of error terms
\begin{equation}
	{\nn}^{\Delta\tau/d} \sigma_{\nn,p_0,\tau_0}(f)
\end{equation}
is in $\ell_{p_1}$ with respect to a Haar-type measure on $\N$.

Using this characterization along with standard embedding properties of approximation spaces \cite[Chapter 7]{devore93}, we derive our result. In particular,  \eqref{eq:besov_char} together with the aforementioned embedding implies that
\begin{equation}
	b_{p_1,p_1}^{\tau_0+\Delta\tau}
	\subset
	A_{\infty}^{\Delta\tau/d} (b_{p_0,p_0}^{\tau_0}).
\end{equation}
Similarly, we have that
\begin{equation}
	A_{\infty}^{\Delta\tau/d+\epsilon} (b_{p_0,p_0}^{\tau_0}) 
	\subset
	b_{p_1,p_1}^{\tau_0+\Delta\tau}.
\end{equation}
The fact that the continuous-domain Besov spaces are isomorphic to Besov sequence spaces (Proposition \ref{prop:besov_char})  completes the proof.
\end{proof}

\subsection{The Compressibility of Generalized L\'evy Processes} \label{subsec:compressibility}

The compressibility of a function quantifies the speed of convergence of its $\nn$-term approximation  in a wavelet basis. 

\begin{definition} \label{def:compressibility}
For a generalized function $f \in \dot{B}_{p_0,p_0}^{\tau_0}(\T^d)$, we define its $(p_0,\tau_0)$-\emph{compressibility} as
\begin{equation} \label{eq:kappa}
	\kappa_{p_0,\tau_0} (f) 
	= 
	\sup 
	\setb{\lambda \geq 0}
	{  
	\sup_{\nn \in \Z_{\geq 0}} 
	{(\nn+1)}^{\lambda} 
	\norm{ f - \Sigma_{\nn,p_0,\tau_0} (f) }_{\dot{B}_{p_0,p_0}^{\tau_0}(\T^d)} 
	< 
	\infty 
	}. 
\end{equation}
\end{definition}

The quantity $\kappa_{p_0,\tau_0}(f) \in [0,\infty]$ is well-defined and nonnegative for $f \in \dot{B}_{p_0,p_0}^{\tau_0}(\T^d)$. The higher $\kappa_{p_0,\tau_0}(f)$ is, the more $f$ is compressible. If the approximation error has a faster-than-algebraic decay, then $\kappa_{\tau_0,p_0}(f) = + \infty$. 

\begin{remark}
One can interpret compressibility in terms of approximation spaces.  The sequence 
$E_n(f) := \norm{ f - \Sigma_{\nn,p_0,\tau_0} (f) }_{\dot{B}_{p_0,p_0}^{\tau_0}(\T^d) } $
specifies how well $f$ can be approximated by wavelets.  Weighted norms of this sequence are commonly used to characterize approximation properties, and such norms specify approximation spaces $A_{q}^{\lambda}$. For example
\begin{equation}
	\norm{f}_{A_q^{\lambda}}
	:=
	\begin{cases}
	\parenth{ 
	\sum_{n=1}^{\infty} 
	\parenth{n^{\lambda} E_{n-1}(f)}^{q} \frac{1}{n} }^{1/q},
	& 0  < q < \infty \\
	\sup \parenth{n^{\lambda} E_{n-1}(f)},
	& q=\infty.
	\end{cases}
\end{equation}  
Therefore, $\kappa_{p_0,\tau_0}(f)$ is the largest $\lambda$ such that $f \in A_{\infty}^{\lambda}$. 
\end{remark}

We first give a general result on the $\nn$-term approximation of a sparse process $s = \Lop^{-1} w$ in 
$\dot{B}_{p_0,p_0}^{\tau_0}(\T^d)$. 

\begin{proposition} \label{prop:boundsigmastochastic}
We consider a stochastic process $s = \Lop^{-1} w$ with Blumenthal-Getoor indices 
$0 \leq \beta' \leq \beta \leq 2$ and order $\gamma \geq 0$.
If $\tau_0 \in \R$ and $0<p_0 \leq \infty$ satisfy the relation
\begin{equation} \label{eq:conditiontaup}
	\gamma 
	> 
	\tau_0 + d - \frac{d}{\max \parenth{p_0, \beta}},
\end{equation}
then we have the order relation
\begin{equation} \label{eq:boundsigma}
	\sigma_{\nn,p_0,\tau_0} ( s ) 
	\leq 
	C 
	n^{\gamma / d + 1 / \beta - 1} 
	\lVert s \rVert_{\dot{B}_{p_1,p_1}^{\tau_1}(\T^d)}
\end{equation}
almost surely, 		
where $\tau_1$ and $p_1$ are set as
\begin{equation}
	\tau_1 - \tau_0 
	= 
	 \gamma 
	+ 
	\frac{d}{\beta} 
	- 
	d 
	=
	\frac{d}{p_1} - \frac{d}{p_0}
\end{equation}
and $C$ is a (deterministic) constant.
\end{proposition}
\begin{proof}
Due to Corollary \ref{coro:besovregus}, Condition \eqref{eq:conditiontaup} ensures that $s \in \dot{B}_{p_0,p_0}^{\tau_0} (\T^d)$. 
If $s \notin \dot{B}_{p_1,p_1}^{\tau_1} (\T^d)$, then $\lVert s \rVert_{ \dot{B}_{p_1,p_1}^{\tau_1} (\T^d)} = \infty$ so that \eqref{eq:boundsigma} is obvious. 
Otherwise, we apply 
Part i) of Theorem \ref{theo:NtermBesov} with 
$\Delta \tau 
= 
(\tau_1 - \tau_0) 
= 
(\gamma/d + 1/\beta -1) $ 
to deduce the result. 
\end{proof}

In particular, we deduce from Proposition \ref{prop:boundsigmastochastic} that, under Condition \eqref{eq:conditiontaup}, one has almost surely  that
\begin{equation}
	\sigma_{\nn,p_0,\tau_0} ( s ) \underset{n\rightarrow\infty}{\longrightarrow} 0.
\end{equation}

Based on our preliminary work, we now obtain new results on the compressibility of processes $s$ that are solutions of \eqref{eq:ls=w}. We split the results into two cases: the Gaussian processes and the sparse processes.

\begin{theorem}[Compressibility of Gaussian processes] \label{theo:compressGaussian}
Let $s = \Lop^{-1} w_{\mathrm{G}}$ be a Gaussian process of order 
$\gamma$. For $0 < p_0 \leq \infty$  and if  $\tau_0 \in \R$ satisfies
\begin{equation} \label{eq:admissibleGauss}
	\gamma 
	> 
	\tau_0 + \frac{d}{2},
\end{equation}
then we have, almost surely, that
\begin{equation}
	\kappa_{p_0,\tau_0} (s) 
	= 
	\frac{\gamma - \tau_0}{d} - \frac{1}{2}.
\end{equation}
\end{theorem}
\begin{proof}
Fix $0< \epsilon < (\gamma - d/2 - \tau_0)$. Then, according to Corollary \ref{coro:besovregus},  $s \in B_{p_1,p_1}^{\tau_1}(\T^d)$  with 
\begin{equation}
	\tau_1 
	= 
	\gamma - \frac{d}{2} - \epsilon 
	\text{ and } 
	\tau_1 - \tau_0 
	= 
	d\parenth{\frac{1}{p_1} - \frac{1}{p_0}}.
\end{equation}
Then, Proposition \ref{prop:boundsigmastochastic} implies that
\begin{equation}
	\sigma_{n,p_0,\tau_0}(s) 
	\leq 
	C n^{\frac{1}{p_1} - \frac{1}{p_0}} 
	\lVert s \rVert_{\dot{B}_{p_1,p_1}^{\tau_1}(\T^d)}
\end{equation}
with $C>0$ a constant. This shows that 
\begin{equation}
	\kappa_{p_0,\tau_0} 
	\geq 
	\frac{1}{p_1} - \frac{1}{p_0}
	= 
	\frac{\gamma - \tau_0}{d} - \frac{1}{2} - \frac{\epsilon}{d}.
\end{equation}
This is valid for $\epsilon$ arbitrarily small, implying that 
$\kappa_{p_0,\tau_0} (s) 
\geq  
\parenth{(\gamma - \tau_0) / d - 1/2}
$. 
		 
We also know from \eqref{eq:GaussBesovnegative} that $s$ is not in $\dot{B}_{p_1,p_1}^{\tau_1}(\T^d)$ for 
$\tau_1 
\geq 
(\gamma - d /2)$ 
and $p_1$ given by 
$(\tau_1 - \tau_0 )
= 
d\parenth{1/p_1 -1/p_0}
$. 
The converse of Theorem \ref{theo:NtermBesov}, Part ii), implies that there exists no constant $C>0$ such that $\sigma_{n,p_0,\tau_0}(s) \leq C n^{-(\gamma - \tau_0 -  d / 2)}$.
Therefore, 
$\kappa_{p_0,\tau_0}(s)  
\leq  
\parenth{(\gamma -\tau_0)/d - 1/2}
$, 
which finishes the proof.
\end{proof}

\begin{theorem}[Compressibility of sparse processes] \label{theo:compressSparse}
Let $s = \Lop^{-1} w$ be a sparse process of order $\gamma$ and Blumenthal-Getoor indices $0 \leq \beta' \leq \beta \leq 2$. 
We assume that $\tau_0 \in \R$ and $0<p_0 \leq \infty$ satisfy
\begin{equation}
	\gamma 
	> 
	\tau_0 + d - \frac{d}{p_0}. 
	\label{eq:admissibleSparse}
\end{equation}
\begin{itemize}
\item 
If $\beta = 0$, then, almost surely,
\begin{equation}
	\kappa_{p_0,\tau_0} (s) 
	= 
	+ \infty.
\end{equation}
\item If $\beta > 0$, then, almost surely,
\begin{equation} \label{eq:kappasparse}
	\frac{\gamma - \tau_0}{d} + \frac{1}{\beta} - 1 \leq 	\kappa_{p_0,\tau_0} (s) \leq \frac{\gamma - \tau_0}{d} + \frac{1}{\beta'} - 1	
\end{equation}
\end{itemize}
\end{theorem}	
\begin{proof}
We first assume that $\beta > 0$.
We proceed as in the Gaussian case. Condition \eqref{eq:admissibleSparse} allows us to consider 
$0< \epsilon < (\gamma - \tau_0) + d( 1 / p_0 - 1)$ 
(and also implies that 
$s \in \dot{B}_{p_0,p_0}^{\tau_0}(\T^d)$, 
due to Corollary \ref{coro:besovregus}).
We set $\tau_1, p_1$ such that
\begin{equation}
	\tau_1 
	= 
	\gamma + \frac{d}{\beta} - d - \epsilon 
	\quad \text{and} \quad 
	\frac{d}{p_1} - \frac{d}{p_0} 
	= 
	\gamma + \frac{d}{\beta} - d - \tau_0 - \epsilon. 
\end{equation}
We claim that $s \in \dot{B}_{p_1,p_1}^{\tau_1}(\T^d)$. Based on Corollary \ref{coro:besovregus}, this is true if 
\begin{equation}
	\tau_1 
	< 
	\gamma 
	+ 
	\frac{d}{\max \parenth{p_1, \beta}} 
	- 
	d.
\end{equation}
Then, we remark that
\begin{equation}
	\frac{1}{p_1} 
	= 
	\frac{1}{\beta} 
	+ 
	\parenth{ \gamma - \tau_0 + \frac{d}{p_0} - d - \epsilon} 
	> 
	\frac{1}{\beta},
\end{equation}
so that $p_1 \leq \beta$. Moreover,
\begin{align}
	\tau_1 
	&= 
	\gamma + \frac{d}{\beta} - d - \epsilon 
	< 
	\gamma + \frac{d}{\beta} - d
	=  
	\gamma + \frac{d}{\max\set{p_1, \beta}} - d,
\end{align}
as expected. Due to Theorem \ref{theo:NtermBesov}, Part i), one deduces that 
\begin{equation}
	\kappa_{p_0,\tau_0}(s) 
	\geq  
	\frac{\gamma - \tau_0}{d} + \frac{1}{\beta} - 1 - \frac{\epsilon}{d}
\end{equation} 
for $\epsilon$ arbitrarily small and, therefore, the lower bound of \eqref{eq:kappasparse} is obtained. 

For the upper bound, we proceed as for the Gaussian case. For $\tau_1 > \gamma + d / \beta' - d$ and $p_1$ satisfying $\tau_1 - \tau_0 = d(1 / p_1 - 1 / p_0)$, we know from Corollary \ref{coro:besovregus} that $s$ is almost surely not in $\dot{B}_{p_1,p_1}^{\tau_1} (\T^d)$. Therefore, Theorem \ref{theo:NtermBesov}, Part ii) implies that there is no constant $C>0$ such that $\sigma_{n,p_0,\tau_0}(s) \leq Cn^{- ( (\gamma -\tau_0)/d + (1/\beta' - 1)}$. Thus. $\kappa_{p_0,\tau}(s) \leq ((\gamma - \tau_0) / d + 1 / \beta' - 1)$, concluding the proof in this case. \\

For $\beta = 0$, we know that 
$s \in \dot{B}_{p,p}^{\tau}(\T^d)$ 
when 
$\tau < (\gamma + d/p - d)$ 
(Corollary \ref{coro:besovregus}).
For $p_1 < p_0$, we define 
$\tau_1 = (\tau_0 + d / p_1 - d / p_0)$. 
Then, using \eqref{eq:admissibleSparse}, we have that
\begin{equation}
	\tau_1 
	= 
	\parenth{\tau_0 - \frac{d}{p_0} } 
	+ 
	\frac{d}{p_1} 
	< 
	\gamma + \frac{d}{p_1} - d
\end{equation}
and, therefore, $s \in \dot{B}_{p_1,p_1}^{\tau_1}(\T^d)$. With Theorem \ref{theo:NtermBesov}, we deduce that 
$\kappa_{p_0,\tau_0} (s) 
\geq 
(1/p_1 - 1/p_0)$ 
for any $p_1$ arbitrarily small. Finally, this means that $\kappa_{p_0,\tau_0}(s) = + \infty$ almost surely. 
\end{proof}

Note that Theorem \ref{theo:compressGaussian} gives the exact value of $\kappa_{p_0,\tau_0}(s)$ in the Gaussian case, while we only provide lower and upper bounds for sparse processes in Theorem \ref{theo:compressSparse} in general. It is important to remark however that the compressibility is known as soon as $\beta = \beta'$, in particular when $\beta = 0$.  

\section{Discussion and Examples} \label{sec:discuss}

In this section, we focus on the $L_2$-compressibility that is obtained for $p_0 = 2$ and $\tau_0 = 0$.  
This case is of special interest, in particular for signal-processing applications: the quantity 
$\kappa (f) = \kappa_{0,2}(f)$ 
measures the approximation error in the $L_2$-sense. We therefore reformulate Theorems \ref{theo:compressGaussian} and \ref{theo:compressSparse} for this case. 
	
\begin{corollary} [$L_2$-compressibility of Gaussian and sparse processes] \label{coro:L2compressible}
Let $s = \Lop^{-1} w$ be a process of order 
$\gamma > d/2$ and of  Blumenthal-Getoor indices $0\leq \beta' \leq \beta \leq 2$.
\begin{itemize}
\item 
If $w = w_{\mathrm{G}}$ is a Gaussian white noise, then, almost surely,
\begin{equation}
	\kappa ( s ) 
	= 
	\frac{\gamma}{d} - \frac{1}{2}.
\end{equation}
\item 
If $w$ is sparse  with $\beta > 0$, then, almost surely,
\begin{equation}
	\kappa ( s ) 
	= 
	+ \infty
\end{equation}
\item 
If $w$ is sparse   with $\beta > 0$, then, almost surely,
\begin{equation} \label{eq:kappasparsebis}
	\frac{\gamma}{d} 
	+ 
	\frac{1}{\beta} - 1
	\leq 	
	\kappa ( s )  
	\leq   
	\frac{\gamma}{d} 
	+ 
	\frac{1}{\beta'} - 1 .
\end{equation}
\end{itemize}
\end{corollary}

We introduce some classical families of L\'evy white noises with their indices in Table \ref{table:noises}.
Precise definitions can be found in the proposed references.
The compressibility of the associated $\gamma$-admissible processes is deduced from Corollary \ref{coro:L2compressible}. 
Note that $\beta = \beta'$ in all the considered examples, hence the compressibility is completely determined by \eqref{eq:kappasparsebis}. \\

\begin{table} [!ht] 
\footnotesize  
\centering
\caption{Compressibility of Gaussian and sparse processes  of order $\gamma > d/2$  based on specific L\'evy white noises}
\label{table:noises}
\begin{tabular}{lcccc} 
\hline
\hline 
White noise $w$ 
& Parameter 
& $\lexp(\xi)$ 
&  $\beta$ 
& Compressibility 
\\
\hline\\[-1ex]
Gaussian 
&$\sigma^2 >0$ 
& $- {\sigma^2 \xi^2}/{2}$ 
& $2$    
& $\gamma - \frac{d}{2}$ 
\\
Cauchy \cite{samoradnitsky94}
& --- 
& $ - \abs{\xi} $ 
& $1$ 
& $\gamma$ 
\\
S$\alpha$S  \cite{samoradnitsky94} 
& $\alpha\in (0,2)$ 
& $ - \abs{\xi}^{\alpha}$ 
& $\alpha$ 
& $ \gamma + d/\alpha - d$ 
\\
Compound Poisson \cite{Unser2011stochastic} 
& $\lambda>0, \mathbb{P}_{\mathrm{J}}$ 
& $\exp( \lambda ( \widehat{\mathbb{P}}_{\mathrm{J}}(\xi) - 1))$ 
& $0$ 
& $\infty$ 
\\
Laplace \cite{Koltz2001laplace} 
& --- 
& $-\log (1 + \xi^2) $ 
& $0$ 
& $\infty$ 
\\
Inverse Gaussian \cite{Barndorff1997processes}
& --- 
& ---
& $1/2$ 
& $\gamma + d$ 
\\
\hline
\hline
\end{tabular}
\end{table}

We finish with some important remarks.
	
\begin{itemize}
\item 
When $\beta > 0$, we have only obtained upper and lower bounds for the compressibility of a generalized L\'evy process. It would be interesting to investigate the case when $\beta' < \beta$, for which the compressibility is unknown. 
\item 
In the Gaussian case,  one has the exact value of $\kappa(s)$. This is due to the fact that the converse result for the Besov regularity of Gaussian white noises are known \cite{veraar2010regularity}. A fundamental consequence is the following: If $s = \Lop^{-1} w$ is a sparse process of order $\gamma$ and $s_{\mathrm{G}} = \Lop^{-1} w_{\mathrm{G}}$ is a Gaussian process corresponding to the same operator $\Lop$, then we have almost surely that
\begin{equation}
	\kappa(s) \geq \kappa ( s_{\mathrm{G}} ).
\end{equation}
Simply stated, sparse processes are more compressible than Gaussian processes. This finally gives a functional justification for the terminology of \emph{sparse} processes introduced in \cite{Unser2014sparse}. 

\item We assume now that $\beta = \beta'$, hence the compressibility is fully characterized. 
For a fixed $\gamma$, the smaller $\beta$ is, the more compressible is the process. At the limit, for $\beta = 0$, the error decays faster than any polynomial. This is achieved by compound-Poisson processes and also by Laplace processes, that are therefore the sparsest processes. These results are coherent with recent works of H. Ghourchian et al. in their recent works on the entropy of sparse processes~\cite{Ghourchian2017compressible}. Here, the authors have shown that, among S$\alpha$S and compound Poisson white noises, the less sparse is the Gaussian white noise, and the sparsest are the compound Poisson white noises. Even though their notion of sparsity is different, we believe that their results have a strong connection with ours.

\item 
For fixed $\beta = \beta' > 0$, the compressibility of the process increases with $\gamma$. This is a very reasonable outcome as it is well known that smoother functions are more compressible.  
\end{itemize}
	

\bibliographystyle{plain}
\bibliography{Nterm_arxiv}

\end{document}